\documentclass[12pt]{article}
\RequirePackage[colorlinks,citecolor=blue,urlcolor=blue,linkcolor=blue]{hyperref}
\hypersetup{
colorlinks = true,
citecolor=blue,
urlcolor=blue,
linkcolor=blue,
pdfauthor = {Alexey Kuznetsov, Justin Miles},
pdfkeywords = {Gaver-Stehfest algorithm, rate of convergence, inverse Laplace transform, Lambert W-function, generating functions, exponential convergence},
pdftitle = {On the rate of convergence of the Gaver-Stehfest algorithm},
pdfpagemode = UseNone
}
\usepackage{graphicx,xspace,colortbl}
\usepackage{epstopdf}
\usepackage{amsmath,amsthm,amsfonts,amssymb}
\usepackage{color}
\usepackage{enumerate}
\usepackage{fancybox}
\usepackage{epsfig}
\usepackage{subfig}
\usepackage{pdfsync}
\usepackage{comment}
    \def\qed{\hfill$\sqcap\kern-8.0pt\hbox{$\sqcup$}$\\}
    \def\beq{\begin{eqnarray}}
    \def\eeq{\end{eqnarray}}
    \def\beqq{\begin{eqnarray*}}
    \def\eeqq{\end{eqnarray*}}

\DeclareMathOperator{\re}{Re}
\DeclareMathOperator{\im}{Im}

    \def\r{{\mathbb R}}
    	
    \def\d{{\textnormal d}}
    \def\i{{\textnormal i}}

\newtheorem{theorem}{Theorem}
\newtheorem{lemma}{Lemma}

\theoremstyle{definition}

\title{On the rate of convergence of the Gaver-Stehfest algorithm}

\author{
Alexey Kuznetsov and Justin Miles 
\footnote{Dept. of Mathematics and Statistics,  York University,
4700 Keele Street, Toronto, ON, M3J 1P3, Canada.   Email: akuznets@yorku.ca,  justinm@mathstat.yorku.ca}
 }

\begin{document}

\maketitle

\begin{abstract}
The Gaver-Stehfest algorithm is widely used for numerical inversion of Laplace transform. In this paper we provide the first rigorous study of the rate of convergence of the Gaver-Stehfest algorithm. We prove that Gaver-Stehfest approximations converge exponentially fast if the target function is analytic in a neighbourhood of a point and they converge at a rate $o(n^{-k})$ if the target function is $(2k+3)$-times differentiable at a point.   
\end{abstract}

\vspace{0.25cm}

{\vskip 0.15cm}


\section{Introduction and main results}


The Gaver-Stehfest algorithm for numerical inversion of Laplace transform has a long history. In 1966 Gaver  \cite{Gaver1966} has introduced simple (but rather slowly convergent) approximations for the inverse Laplace transform, and in 1970 Stehfest \cite{Stehfest1970,Stehfest2} has applied convergence accelleration to Gaver's approximation and thus the Gaver-Stehfest algorithm was born. The algorithm turned out to be very popular with practitioners due to a number of desirable properties: it is linear, it is exact for constant functions, all the coefficients can be computed explicitly and, most importantly, the algorithm
does not require the use of complex numbers, as it needs the values of the Laplace transform only on the positive real line. The price one has to pay for this latter feature is that the algorithm requires high-precision arithmetic for its implementation.

Let us present the Gaver-Stehfest algorithm. We start with a locally integrable function $f:(0,\infty)\mapsto \r$, such that its Laplace transform 
\begin{equation}
F(z):=\int_0^{\infty} e^{-zx} f(x) \d x 
\end{equation}
is finite for all $z>0$. We want to solve the following inverse problem: given the values of $F(z)$ for $z>0$, compute the value of $f(x)$ at a given point $x>0$. 
Gaver-Stehfest approximations are given by 
\begin{equation}\label{def_Gaver_Stehfest_apprx}
f_n(x):=\ln(2) x^{-1}
 \sum\limits_{k=1}^{2n} a_k(n) F\left(k  \ln(2) x^{-1} \right), \;\;\; n\ge 1, \; x>0,
\end{equation}
where
$$
a_k(n):=
  \frac{(-1)^{n+k}}{n!} \sum\limits_{j=[(k+1)/2]}^{\min(k,n)} j^{n+1}
{n \choose j}{2j \choose j}{j \choose k-j}, \;\; \; 1\le k \le 2n.  
$$
In \cite{Kuznetsov_2013} several conditions for convergence of $f_n(x_0)$ were established. It was proved that if $f$ has bounded variation or is 
H\"older continuous in a neighbourhood of $x_0>0$, then $f_n(x_0)$ converge to $(f(x_0+)+f(x_0-))/2$ as $n\to \infty$. The question of the rate of convergence was left open, and until now there were no rigorous results about the rate of convergence of the Gaver-Stehfest algorithm (although there were many numerical studies of the convergence of the algorithm -- see \cite{Whitt06,Davies1979,Jacquot1983,Masol2010,Valko2004629} and the references therein). It is the goal of this paper to provide the first rigorous treatment of the rate of convergence of the Gaver-Stehfest algorithm.  We establish the following two results:

\begin{theorem}\label{thm_main}
Assume that $f$ is analytic in a neighborhood of $x_0>0$. Then there exists $c>0$ such that 
\begin{equation}\label{exp_convergence}
f_n(x_0)=f(x_0)+O(e^{-cn}), \;\;\; n\to +\infty. 
\end{equation}
\end{theorem}

\begin{theorem}\label{thm2}
Assume that $m\ge 5$ and $f$ is $m$-times differentiable at $x_0>0$. Set $k=[(m-3)/2]$. 
Then
\begin{equation}\label{power_convergence}
f_n(x_0)=f(x_0)+o(n^{-k}), \;\;\; n\to +\infty.
\end{equation}
\end{theorem}

The above two theorems lead to two natural problems: determine the largest values of $c$ and $k$ in \eqref{exp_convergence} and \eqref{power_convergence}. The first problem, that is trying to determine the largest value of $c$ in \eqref{exp_convergence} is likely to be very hard and we do not have any intuition as to what the answer may be. For the second problem we do have the following conjecture, supported by a number of numerical experiments

\vspace{0.25cm}
\noindent
{\bf Conjecture:} 
{\it 
If $f$ is $m$-times differentiable at $x_0>0$ then
$$
f_n(x_0)=f(x_0)+O(n^{-m}), \;\;\; n\to +\infty.
$$
}

We arrived at this conjecture by investigating the rate of convergence of the Gaver-Stehfest approximations to functions of the form
\begin{equation}\label{function_f}
f(x)=(x-1)^{\alpha} e^{-\beta x} \times {\mathbf 1}_{\{ x> 1\}},
\end{equation}
where $\alpha>0$ and $\beta \in {\mathbb C}$ with $\re(\beta)\ge 0$. This function clearly satisfies $f(1)=0$ and is $m$-times differentiable at $x=1$ for any integer $m<\alpha$. The corresponding Laplace transform is 
easily computed explicitly
$$
F(z)=\int_0^{\infty} f(x) e^{-zx} \d x=\int_1^{\infty} (x-1)^{\alpha} e^{-(\beta+z)x} \d x =\Gamma(\alpha+1) (\beta+z)^{-\alpha-1} e^{-\beta-z}, \;\;\; z>0. 
$$
To find the optimal value of $k$ in \eqref{power_convergence} we computed Gaver-Stehfest approximations $f_n(1)$ for $1\le n \le 300$ (using high-precision arithmetic) and then we used linear regression to compute $k$ that provides the best fit for $\ln|f_n(1)| \sim C-k \ln(n)$, $1\le n \le 300$. This procedure was repeated many times with different values of parameters $\alpha$ and $\beta$ and the above conjecture seems to hold true for all functions of the form \eqref{function_f}.

The paper is organized as follows. In Section \ref{section_preliminary_results} we 
state and prove Theorem \ref{thm3}, which is the foundation of our approach. In Section \ref{section_proof_thm1} we prove Theorem \ref{thm_main} and in Section \ref{section_proof_thm2} we prove Theorem \ref{thm2}.


\section{Preliminary results}\label{section_preliminary_results}


Let us review some properties of the Lambert W-function, which will be needed later.
The principal branch of the Lambert W-function, denoted by $W(z)$,
 is an analytic function in the neighborhood of $z=0$ that satisfies $W(z)\exp(W(z))=z$.   
  It is well-known   \cite{Corless96} that $W$ is analytic in ${\mathbb C}\setminus (-\infty,-e^{-1}]$, 
and it has the following Taylor series at $z=0$ (see formula (3.1) in \cite{Corless96})
\beq\label{W_series}
W(z)=\sum\limits_{n\ge 1} (-n)^{n-1} \frac{z^n}{n!}, \;\;\; \vert z \vert<1/e,
\eeq
and a branching singularity at $z=-1/e$
\beq\label{expansion_W_near_1_over_e}
W(z)= -1+p-\frac{p^2}{3}+\frac{11}{72} p^3-\frac{43}{540} p^4+ \frac{769}{17280}p^5 
+...=\sum_{n\ge 0} \mu_n p^n, 
\eeq
where $p=\sqrt{2(1+ez)}$ and the series converges for $|p|<\sqrt{2}$ (see formula (4.22) in \cite{Corless96}).
The coefficients $\mu_n$ are certain rational numbers that can be computed recursively (see formulas (4.23) and (4.24) in \cite{Corless96}).

We define
\begin{equation}\label{H_in_terms_of_W}
H(z):=-\left(z\frac{\d}{\d z} \right)^2 W(z)=-\frac{W(z)}{(1+W(z))^3}. 
\end{equation}
The second equality follows from the identity $zW'(z)=W(z)/(1+W(z))$, which can be easily derived from
the functional equation $W(z)\exp(W(z))=z$. 
Since $W$ is analytic in ${\mathbb C}\setminus (-\infty,-e^{-1}]$ and satisfies $W(0)=0$, it is clear from 
\eqref{H_in_terms_of_W} that $H$ is also analytic in ${\mathbb C}\setminus (-\infty,-e^{-1}]$ and satisfies $H(0)=0$.

From  \eqref{expansion_W_near_1_over_e} and \eqref{H_in_terms_of_W}  we derive series representation
\begin{equation}\label{Hz_series_expansion}
H(z)=p^{-3}-\frac{11}{24}p^{-1}-\frac{4}{135}-\frac{1}{1152} p-\frac{31}{405} p^2-\dots
= p^{-3}-\frac{11}{24}p^{-1} + \sum\limits_{n\ge 0} c_n p^n,
\end{equation}
where, as above, $p=\sqrt{2(1+ez)}$ and the series converges for $|p|<\sqrt{2}$. The coefficients $c_n$ in \eqref{Hz_series_expansion} are certain rational numbers that can be computed recursively using values of $\mu_n$. We define the following two functions in terms of coefficients $c_n$: 
\begin{align}
\label{def_Au}
A(u)&:=\frac{1}{2\sqrt{2}}-\frac{11}{24 \sqrt{2}}(1+u)+ \sum\limits_{n\ge 0}  c_{2n+1} 2^{n+1/2} (1+u)^{n+2},\\
\label{def_Bu}
B(u)&:=\sum\limits_{n\ge 0} c_{2n} 2^n (1+u)^n. 
\end{align}
Since the series in \eqref{Hz_series_expansion} converges for $|p|<\sqrt{2}$, we conclude that the series \eqref{def_Au} and \eqref{def_Bu} converge for $|1+u|<1$, thus functions $A$ and $B$ are analytic in the disk $D_{1}(-1)$: here and everywhere else in this paper we will denote  
$$
D_r(a):=\{z\in {\mathbb C} \; : \; |z-a|<r\}, 
$$
for $r>0$ and $a\in {\mathbb C}$.  By construction we have an identity
 \begin{equation}\label{eqn_H_decomposition}
H(z/e)=(1+z)^{-3/2}A(z)+B(z), 
\end{equation}
which is valid for $z\in D_1(-1) \setminus (-\infty, -1]$.

Next, given a function $f$ and $x_0>0$ we define
\begin{equation}\label{def_tilde_f}
\tilde f(v):=\frac{f(x_0 \log_{1/2}((1+v)/2))}{1+v}+\frac{f(x_0 \log_{1/2}((1-v)/2))}{1-v}, \;\;\; -1<v<1, 
\end{equation}
and 
\begin{equation}\label{def_phix}
\phi(x):=\frac{1}{\pi\sqrt{1-x}}  \int_0^{\frac{\pi}{2}}\tilde f(\sqrt{x} \sin(y)) \d y, \;\;\; 0\le x < 1. 
\end{equation} 
We also define $w(z):=ze^{z+1}$ and 
\begin{equation}\label{def_Delta3}
\Lambda(w)=\Lambda(w;\sigma):=\int_0^{\sigma} (1+w(1-x))^{-\frac{3}{2}} A(w(1-x)) \phi(x) \d x, 
\end{equation}
where $\sigma \in (0,1)$ and $A$ is defined in \eqref{def_Au}. For every $\sigma \in (0,1)$ the function 
$w\mapsto \Lambda(w;\sigma)$ is well-defined for $w \in D_{\delta}(-1)\setminus (-\infty,-1]$ for some $\delta=\delta(\sigma)>0$ small enough.  Finally, for $\epsilon>0$ we denote
\begin{equation}\label{def_D_epsilon}
{\mathcal D}_{\epsilon}:=D_{1+\epsilon}(0) \setminus D_{\epsilon^{1/4}}(-1) = \{ z\in {\mathbb C} \; : \; |z|<1+\epsilon \; {\textnormal{ and }} \; |1+z|>\epsilon^{1/4} \}.
\end{equation}

The main goal of this section is to establish the following result. 

\begin{theorem}\label{thm3}
Assume that $f(x_0)=0$. 
\begin{itemize}
\item[(i)]  The function 
\begin{equation}\label{def_Delta}
\Delta(z):=\sum\limits_{n\ge 1} f_n(x_0) (-1)^n z^n, 
\end{equation}
is analytic in ${\mathcal D}_{\epsilon}$ for $\epsilon<1/100$. 
\item[(ii)] For any $\sigma \in (0,1)$ the function $\Delta(z)-\Lambda(w(z);\sigma)$ is analytic in $D_{\delta}(-1)$ for some $\delta>0$ small enough. 
\end{itemize}
\end{theorem}

Theorem \ref{thm3} will be our main tool in proving Theorems \ref{thm_main} and \ref{thm2}. We will apply it as follows: suppose we can show that for some $\sigma \in (0,1)$ the function $z\mapsto \Lambda(w(z);\sigma)$ is analytic in $D_{\delta}(-1)$ for some $\delta>0$. Then Theorem \ref{thm3} would imply that $\Delta(z)$ is analytic in $D_R(0)$ for some $R>1$. This latter fact combined with 
\eqref{def_Delta} would prove that the sequence $\{f_n(x_0)\}_{n\ge 1}$ converges to zero exponentially fast. Alternatively, if the function 
$\Lambda(w(z))$ is not analytic in $D_{\delta}(-1)$ for any $\delta>0$, it must have a singularity at $z=-1$, and then the behavior of $\Lambda(w(z))$ at this singularity (for example, the number of times $\Lambda(w(z))$ is differentiable at $z=-1$) would give us information about the singularity of $\Delta(z)$ at $z=-1$, and this informatoin coupled with \eqref{def_Delta} would again lead to estimates on the rate of convergence of the sequence $\{f_n(x_0)\}_{n\ge 1}$ to zero.

\begin{figure}
\centering
\subfloat[]{\label{fig1_p1}\includegraphics[height =6cm]{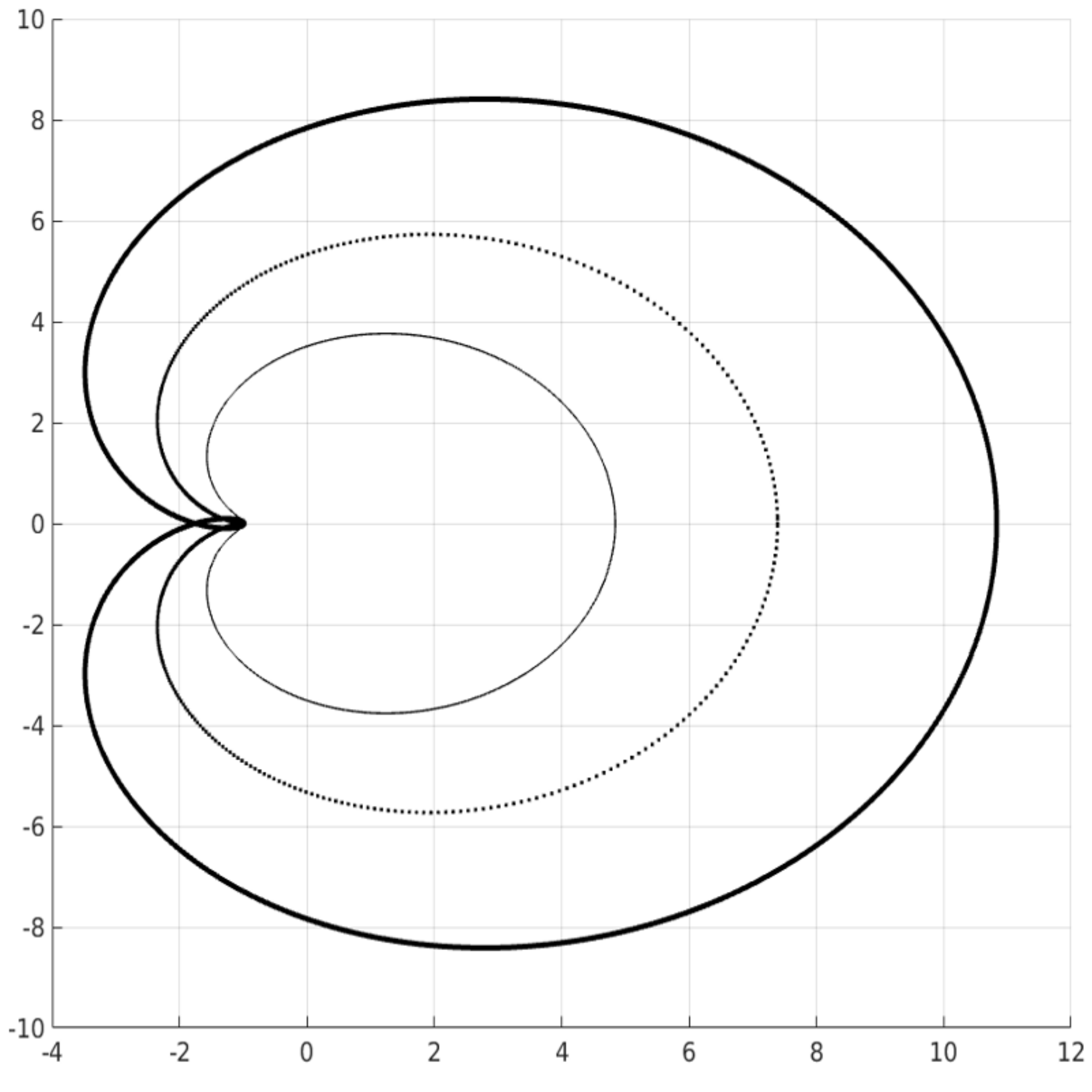}} 
\qquad
\subfloat[]{\label{fig1_p2}\includegraphics[height =6cm]{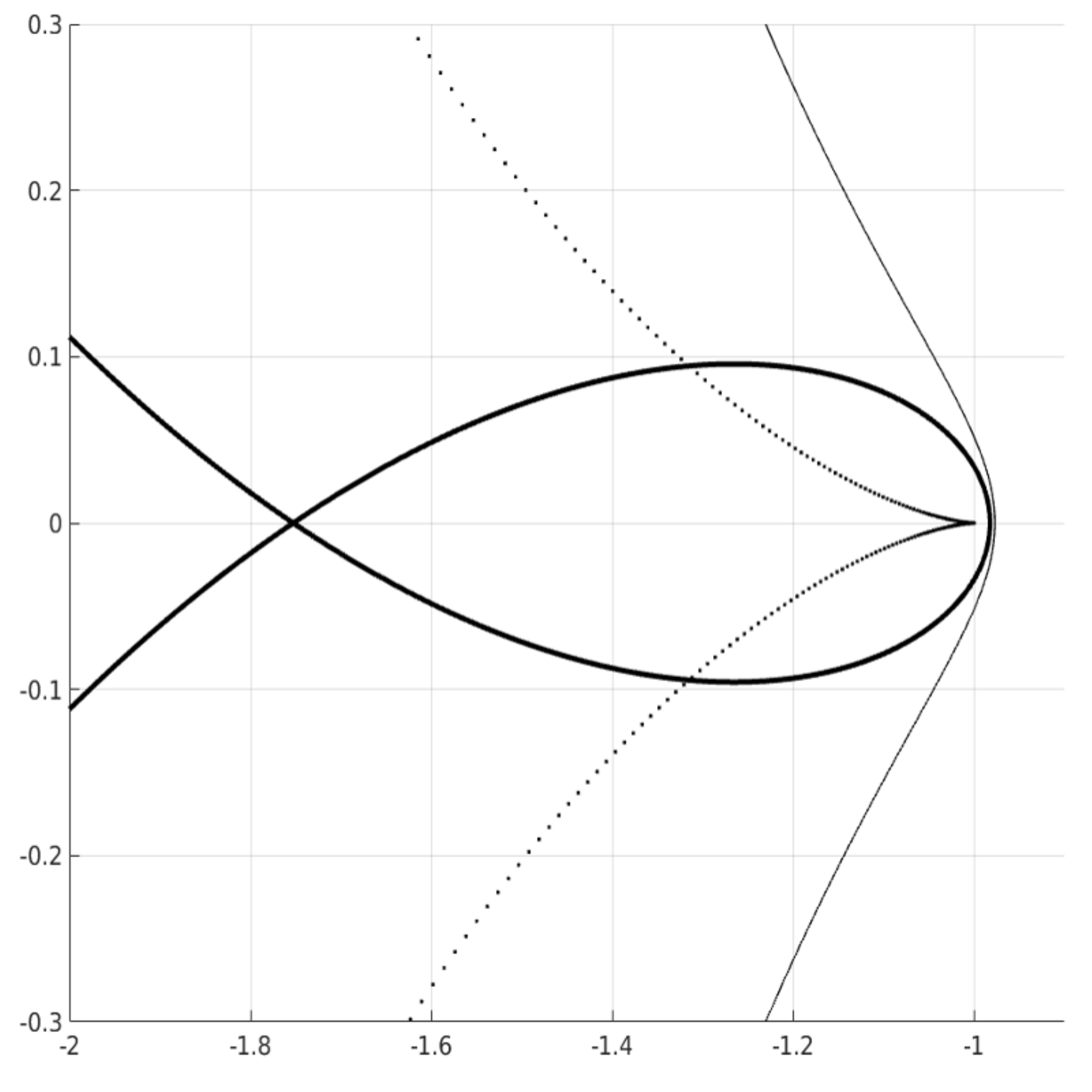}}
\caption{The images of three circles $|z|=0.8$ (thin line), $|z|=1$ (dotted line) and $|z|=1.2$ (thick line) under the map $z\mapsto w=ze^{z+1}$. Figure (b) magnifies the area near $w=-1$ of the figure (a). } 
\label{fig1}
\end{figure}

Before we prove Theorem \ref{thm3}, we need to establish a number of preliminary results. The next technical result 
collects some properties of the map $z\mapsto w=ze^{z+1}$ (see Figure \ref{fig1}).

\begin{lemma}\label{lemma3} Let $w(z)=ze^{z+1}$.
\begin{itemize}
\item[(i)] For any $c \in (0,1)$ there exists $R>1$ such that the function $c\times w(z)$ maps $D_R(0)$ into \\ ${\mathbb C}\setminus (-\infty,-1]$. 
\item[(ii)]
For $\epsilon \in (0,1/100]$ the function $w=w(z)$ maps the domain ${\mathcal  D}_{\epsilon}$
(defined in \eqref{def_D_epsilon}) into 
${\mathbb C}\setminus (-\infty,-1]$.
\end{itemize}
\end{lemma}
\begin{proof}
First we will establish the following 

\vspace{0.2cm}
\noindent 
{\bf Fact:} If $\epsilon  \in (0,1/100]$ and for some $y\in {\r}$ we have $z=-y \cot(y)+\i y \in D(0;1+\epsilon)$, then necessarily $z\in D_{\epsilon^{1/4}}(-1)$. 
\vspace{0.2cm}

To prove this fact we will need the following two inequalities
\begin{align}
\label{inequality1}
&0<1-y\cot(y)< y^2/2 \;\;\;\; {\textrm{ for all }} \; y\in (-1/4,1/4),\\
\label{inequality2}
&1+y^2/6 <y \csc(y)\qquad \;\;\;\; {\textrm{ for all }} \; y\in (-\pi,\pi).
\end{align}
These inequalities can be easily established by examining MacLaurin series of $y\cot(y)$ and $y\sec(y)$. Alternatively, these inequalities follow at once from inequalities (17) and (19) in \cite{Chen_Qi_2004}.

Now, if $\epsilon \in (0,1/100]$ and $z=-y\cot(y)+\i y\in D_{1+\epsilon}(0)$ , then 
$$
|z|^2=y^2 \cot(y)^2 + y^2<(1+\epsilon)^2,
$$
thus $|y|<1+\epsilon$ and from \eqref{inequality2} we find 
$$
(1+y^2/6)^2< y^2 \csc(y)^2=y^2 \cot(y)^2 + y^2<(1+\epsilon)^2,
$$
which implies that $|y|<\sqrt{6 \epsilon}<1/4$. Then applying \eqref{inequality1} we estimate
$$
|z+1|^2=(1-y\cot(y))^2+y^2<y^4/4+y^2<9 \epsilon^2+6 \epsilon=\epsilon^{1/2} \times \epsilon^{1/2} (9\epsilon+6)<\epsilon^{1/2}, 
$$
and this implies $z\in D_{\epsilon^{1/4}}(-1)$. This ends the proof of the Fact above. 

\vspace{0.25cm}

Let us now prove part (i) of Lemma \ref{lemma3}. Since $w(-1)=-1$ and $w$ is an entire (and thus, continuous) function, there exists $\epsilon \in (0,1/100]$ small enough such that $|z+1|<\epsilon^{1/4}$ implies 
$|w(z)+1|<1/c-1$. Take $R=1+\epsilon$ and let $z\in D_R(0)$. If $z=x+\i y$ for $x,y\in \r$, then 
$$
\im(w(z))=e^{x+1}(y\cos(y)+x\sin(y)).
$$
Thus $w(z)\in \r$ if $y=0$ or $x=-y\cot(y)$. For $y=0$ it is easy to see that $w(z)=w(x)\ge -1$, thus 
$-1<c\times w(z)$. If $y\neq 0$ and $w(z) \in \r$ then $z=-y\cot(y)+\i y$. 
Since $z\in D_{1+\epsilon}(0)$, by the Fact above we conclude that $z\in D_{\epsilon^{1/4}}(-1)$, thus 
$|w(z)+1|<1/c-1$, which implies that $-1<c\times w(z)$. Thus if $z \in D_{1+\epsilon}(0)$ and $w(z)$ is real, then necessarily $-1<c\times w(z)$. In other words, the function $c\times w(z)$ maps $D_{1+\epsilon}(0)$ into ${\mathbb C}\setminus(-\infty,1]$. 

It remains to prove part (ii) of Lemma \ref{lemma3}. Let $\epsilon \in (0,1/100]$ and $z=x+\i y \in {\mathcal D}_{\epsilon}$ for $x,y,\in \r$. As we argued above, if $w(z) \in \r$ then either $y=0$ or $x=-y\cot(y)$. In the former case $z$ is real and the minimum of $w(z)$ over real
$z \in {\mathcal D}_{\epsilon}$ is strictly greater than $-1$ (the minimum of $w(x)$ over $x\in \r$ is $-1$ and is achieved at $z=-1$, and $z=-1 \notin {\mathcal D}_{\epsilon}$). In the latter case, we use the Fact above and conclude that $z \in D_{\epsilon^{1/4}}(-1)$, which is impossible since by definition ${\mathcal D}_{\epsilon}$ does not contain points from $D_{\epsilon^{1/4}}(-1)$. Therefore, the function $w(z)$ maps ${\mathcal D}_{\epsilon}$
into ${\mathbb C}\setminus (-\infty,1]$. 
\end{proof}

\begin{lemma}\label{lemma4}
Let $\Omega$ be a compact set in $\r^n$ and $a({\bf x})$ be a continuous function and $b({\bf x})$ an integrable function of
${\bf x}=(x_1,x_2,\dots,x_n) \in \Omega$. 
Assume that $g(z)$ is analytic in the domain $G \subset {\mathbb C}$ and 
$za({\bf x}) \in G$ for all $z\in G$ and ${\bf x} \in \Omega$. 
Then the function 
\begin{equation}
\Phi(z):=\int\limits_{\Omega} g(z a({\bf x})) b({\bf x})\d x_1 \d x_2 \dots \d x_n 
\end{equation}
is also analytic in $G$. 
\end{lemma} 
\begin{proof}
For each ${\bf x}\in \Omega$, the function $z \mapsto g(z a({\bf x})) $ is analytic in $G$. By Cauchy's integral theorem, for each ${\bf x}\in \Omega$ and for any triangle $T$ contained in $G$ we have
$$
\int_T g(z a({\bf x}))  \d z=0.
$$
Since $T$ and $\Omega$ are compact and $g(z a({\bf x}))$ is continuous on $T \times \Omega$, we have $\sup_{(z,x)\in T \times \Omega} |g(z a({\bf x}))|<\infty$. Using this fact and the assumption that $b$ is integrable on $\Omega$, we can apply Fubini's theorem
and conclude that 
$$ 
\int_T \Phi(z) \d z =\int_\Omega \left[\int_T g(z a({\bf x}))  \d z\right] g({\bf x}) \d x_1 \d x_2 \ldots \d x_n =0
$$
for any every triangle $T$ contained $G$.  Morera's Theorem tells us that $\Phi$ is analytic in $G$.
\end{proof}

Next, we define 
\begin{equation}\label{def_Gz}
G(z):=\frac{2}{\pi} \int\limits_0^{\frac{\pi}{2}} H(z\sin(t)^2) \d t,   
\end{equation}
where $H$ was defined in \eqref{H_in_terms_of_W}. As we discussed on page \pageref{H_in_terms_of_W}, $H$ is analytic  in 
${\mathbb C}\setminus (-\infty,-e^{-1}]$ and satisfies $H(0)=0$. This fact and Lemma \ref{lemma4} applied to the integral in 
\eqref{def_Gz} implies that $G$ is also analytic in ${\mathbb C}\setminus (-\infty,-e^{-1}]$ and satisfies $G(0)=0$. 

\begin{lemma}\label{lemma1}
The function $\Delta(z)$ defined by \eqref{def_Delta} has integral representation 
  \begin{equation}\label{Delta_integral_formula}
 \Delta(z)=\int_0^{\infty} G(4e^{-1-u}(1-e^{-u})w(z)) f(x_0u/\ln(2)) \d u
 \end{equation}
and it is analytic in ${\mathcal D}_{\epsilon}$ for $\epsilon<1/100$. 
\end{lemma}
\begin{proof}
From \cite{Kuznetsov_2013} we know that Gaver-Stehfest approximants are given by an integral representation
\begin{equation}\label{f_n_q_n_formula_1}
f_n(x)= \int_0^{\infty}
q_n\left(4e^{-u} (1-e^{-u})\right) f(x u/\ln(2)) \d u.
\end{equation}
where
\begin{equation}\label{def_qnv}
q_n(v):=\sum\limits_{k=1}^n \frac{ k^{n+1}(\tfrac{1}{2})_k}{(n-k)!(k!)^2} (-1)^{n+k} v^k, \;\;\; n\ge 1. 
\end{equation}
Also, from Proposition 2.2 in \cite{Kuznetsov_2013} we find that for $0\le v \le 1$ and $|z|<1/(2e)$ 
\begin{equation}\label{eqn_generating_function_qn}
G\left(vz e^z \right)=\sum\limits_{n\ge 1} q_n(v) (-1)^n z^n.
\end{equation}
Also, from \eqref{def_qnv} we find (using the Binomial Theorem and the trivial estimates $(\tfrac{1}{2})_k<k!$ and $k^{n+1}\le n^{n+1}$) that
\begin{equation}\label{bound_qn}
|q_n(v)| \le v n^{n+1} \sum\limits_{k=1}^n \frac{1}{(n-k)!(k!)} <v\frac{n^{n+1}2^n}{n!},
 \;\;\; {\textnormal{ for all }} \;  0\le v \le 1.
\end{equation}
Thus for every $|z|<1/(4e)$ we have the bound
$$
\sum\limits_{n\ge 1} |q_n(v)|\times |z|^n < C \times v, \;\;\; 0\le v \le 1,
$$
for some $C>0$, so that we can apply the Dominated Convergence Theorem to conclude that 
\begin{align}\label{series_Delta}
\Delta(z)=\sum\limits_{n\ge 1} f_n(x_0) (-1)^n z^n &=
\int\limits_0^{\infty}
\Big[\sum\limits_{n\ge 1} q_n\left(4e^{-u} (1-e^{-u})\right) (-1)^n z^n \Big]f(x_0 u/\ln(2)) \d u
\\
\nonumber
 &=
\int_0^{\infty} G(4ze^{z-u}(1-e^{-u})) f(x_0 u/\ln(2)) \d u.
\end{align}
Thus we have established \eqref{Delta_integral_formula} for $|z|<1/(4e)$. 
The fact that $\Delta(z)$ can be extended to an analytic function in ${\mathcal D}_{\epsilon}$ follows from 
\eqref{Delta_integral_formula}, Lemma \ref{lemma3}(ii), Lemma \ref{lemma4} and the fact that $G(z)$ is an analytic function in 
${\mathbb C}\setminus (-\infty,-e^{-1}]$. 
\end{proof}

Next, we define
\begin{equation}\label{def_Delta1}
 \Delta_1(z)=\Delta_1(z;\sigma):=\int_0^{\sigma} G(e^{-1}  (1-v^2)w(z)) 
 \tilde f(v) \d v,
 \end{equation}
 where $\sigma \in (0,1)$ and $\tilde f$ is defined in \eqref{def_tilde_f}.

\begin{lemma}\label{lemma5}
For any $\sigma \in (0,1)$ there exists $R>1$ such that the function $z\mapsto \Delta(z)-\Delta_1(z;\sigma)$ is analytic in $D_R(0)$. 
\end{lemma}
\begin{proof}
First we compute
\begin{align}\label{Delta_computation1}
\nonumber
 \Delta(z)&=\int_0^{\ln(2)} G(4e^{-1-u}(1-e^{-u})w(z)) f(x_0u/\ln(2)) \d u\\
 & +\int_{\ln(2)}^{\infty} G(4e^{-1-u}(1-e^{-u})w(z)) f(x_0u/\ln(2)) \d u\\ 
 \nonumber
 &= \int_0^1 G(e^{-1} (1-v^2)w(z)) \bigg[ \frac{f(x_0 \log_{1/2}((1+v)/2))}{1+v}+\frac{f(x_0 \log_{1/2}((1-v)/2))}{1-v} \bigg]\d v.
\end{align}
Here we changed variables $u=-\ln((1+v)/2)$ in the integral over $u\in (0,\ln(2))$ and $u=-\ln((1-v)/2)$ in the integral 
over $u\in (\ln(2),\infty)$.  

Next, we define $g(z):=G(z/e)/z$. 
As we pointed out on page \pageref{def_Gz}, the function $G$ is analytic in ${\mathbb C}\setminus(-\infty,-e^{-1}]$
and satisfies $G(0)=0$, thus the function $g$ is analytic in ${\mathbb C} \setminus (-\infty,-1]$.  
From \eqref{def_Delta1} and \eqref{Delta_computation1} we obtain
$$
 \Delta(z)-\Delta_1(z)=w(z) \int_{\sigma}^1 g(w(z)(1-v^2)) b(v) \d v
$$
where 
$$
b(v):=(1-v^2) \tilde f(v)=(1-v)f(x_0 \log_{1/2}((1+v)/2))+(1+v)f(x_0 \log_{1/2}((1-v)/2)).
$$
 According to Lemma \ref{lemma3}(i), there exists $R>1$ such that the 
function $z\mapsto (1-\sigma^2)\times w(z)$ maps $D_R(0)$ into ${\mathbb C}\setminus (-\infty,-1]$. Then for every $v \in (\sigma,1]$ we have 
$$
 (1-v^2) \times w(z) \in {\mathbb C} \setminus (-\infty,-1], \;\;\; 
 {\textrm{ for all }} \; z\in D_R(0). 
$$
Note also that the function $b(v)$ is integrable over
$v\in (\sigma,1]$. Applying Lemma \ref{lemma4}, we conclude that $\Delta(z)-\Delta_1(z)$ is analytic in $D_R(0)$.   
\end{proof}

Next we define 
\begin{equation}\label{def_Delta2}
\Delta_2(z)=\Delta_2(z;\sigma):=\frac{1}{\pi}\int_0^{\sigma} \int_{v^2}^{\sigma^2} \frac{H(e^{-1}(1-x)w(z))}{\sqrt{(1-x)(x-v^2)}}\d x \tilde f(v) \d v
\end{equation}
 where $\sigma \in (0,1)$, $H$ was defined in \eqref{H_in_terms_of_W} and $\tilde f$ was defined in \eqref{def_tilde_f}.

\begin{lemma}\label{lemma6}
For any $\sigma \in (0,1)$ there exists $R>1$ such that the function $z\mapsto \Delta_1(z;\sigma)-\Delta_2(z;\sigma)$ is analytic in $D_R(0)$. 
\end{lemma}
\begin{proof}
We define two sets
\begin{align*}
\Omega_1&:=\Big\{(t,v)\in {\mathbb R}^2 \; : \; 0\le v \le \sigma, \; 0 \le t < \arcsin\Big(\sqrt{\frac{1-\sigma^2}{1-v^2}}\Big)\Big\},\\
\Omega_2&:=\Big\{(t,v)\in {\mathbb R}^2 \; : \; 0\le v \le \sigma, \;  \arcsin\Big(\sqrt{\frac{1-\sigma^2}{1-v^2}}\Big)\le t \le \frac{\pi}{2}\Big\}.
\end{align*}
Using formulas \eqref{def_Gz} and \eqref{def_Delta1} we write
\begin{align*}
\Delta_1(z)&=\frac{2}{\pi}\int\limits_0^{\sigma} \int\limits_0^{\frac{\pi}{2}}H(e^{-1}(1-v^2) \sin(t)^2 w(z))\d t \tilde f(v) \d v\\
\nonumber
&=\frac{2}{\pi}\iint\limits_{\Omega_1}H(e^{-1}(1-v^2) \sin(t)^2 w(z))\d t \tilde f(v) \d v+
\frac{2}{\pi}\iint\limits_{\Omega_2} H(e^{-1}(1-v^2) \sin(t)^2 w(z))\d t \tilde f(v) \d v.
\end{align*}
We change the variable of integration $t\mapsto x=1-(1-v^2)\sin(t)^2$ (so that $\d t=-1/(2\sqrt{(1-x)(x-v^2)}) \d x$) and obtain
$$
\frac{2}{\pi} \iint\limits_{\Omega_2} H(e^{-1}(1-v^2) \sin(t)^2w(z))\d t \tilde f(v) \d v= \Delta_2(z), 
$$
which implies
\begin{equation}\label{eqn_Delta1-Delta2}
\Delta_1(z)-\Delta_2(z)=\frac{2}{\pi}\iint\limits_{\Omega_1}H(e^{-1}(1-v^2) \sin(t)^2w(z))\d t \tilde f(v) \d v.
\end{equation}
Note that on the set $\Omega_1$ we have $(1-v^2)\sin(t)^2\le 1-\sigma^2$, thus we can use the fact that $H(z/e)$ is an analytic function in 
${\mathbb C}\setminus (-\infty,-1]$ and apply Lemma \ref{lemma3}(i) and Lemma \ref{lemma4} to conclude that the integral in the right-hand side of \eqref{eqn_Delta1-Delta2} is an analytic function of $z$ in the disk $D_R(0)$ for some $R>1$. 
\end{proof}

\vspace{0.25cm}
\noindent
{\bf Proof of Theorem \ref{thm3}:}
Part (i) of Theorem \ref{thm3} was established in Lemma \ref{lemma1}. To prove part (ii), it is enough to show that 
the function
$z\mapsto \Delta_2(z;\sqrt{\sigma})-\Lambda(w(z);\sigma)$ is analytic in $D_{\delta}(-1)$ for some $\delta>0$, since
\begin{equation*}
\Delta(z)-\Lambda(w(z);\sigma)=[\Delta(z)-\Delta_1(z;\sqrt{\sigma})]+[\Delta_1(z;\sqrt{\sigma})-\Delta_2(z;\sqrt{\sigma})]
+\Delta_2(z;\sqrt{\sigma})-\Lambda(w(z);\sigma),
\end{equation*}
and both terms in square brackets are analytic in $D_R(0)$ for some $R>1$ (by Lemmas \ref{lemma5} and \ref{lemma6}), thus they are analytic in $D_{\delta}(-1)$ for any $\delta \in (0, R-1]$.

We apply Fubini's Theorem to the double integral \eqref{def_Delta2} and interchange the order of integration to obtain
\begin{align*}
\Delta_2(z;\sqrt{\sigma})&=\frac{1}{\pi} \int_0^{\sigma} \int_0^{\sqrt{x}}
\frac{H(e^{-1}(1-x)w(z))}{\sqrt{(1-x)(x-v^2)}}\tilde f(v) \d v \d x
\\
&=\int_0^{\sigma} H(e^{-1}(1-x)w(z)) \bigg[\frac{1}{\pi \sqrt{1-x}} \int_0^{\sqrt{x}} \frac{\tilde f(v)}{\sqrt{x-v^2}}\d v \bigg] \d x
\\&=
\int_0^{\sigma} H(e^{-1}(1-x)w(z)) \bigg[\frac{1}{\pi\sqrt{1-x}}  \int_0^{\frac{\pi}{2}}\tilde f(\sqrt{x} \sin(y)) \d y\bigg] \d x
\\&= \int_0^{\sigma} H(e^{-1}(1-x)w(z)) \phi(x) \d x.
\end{align*}
In deriving this formula we have changed variable of integration $v=\sqrt{x} \sin(y)$ and used \eqref{def_phix}. 
Next, we apply \eqref{eqn_H_decomposition} to the above identity and obtain
\begin{align*}
\Delta_2(z;\sqrt{\sigma})&=\int_0^{\sigma} H(e^{-1}(1-x)w(z))  \phi(x) \d x\\
&=
\int_0^{\sigma} (1+w(z)(1-x))^{-3/2}A(w(z)(1-x))  \phi(x) \d x+
\int_0^{\sigma} B(w(z)(1-x))  \phi(x) \d x,
\end{align*}
which is equivalent to 
$$
\Delta_2(z;\sqrt{\sigma})-\Lambda(w(z);\sigma)=
\int_0^{\sigma} B(w(z)(1-x))  \phi(x) \d x.
$$
According to the discussion on page \pageref{def_Bu}, the function $B$ is analytic in $D_{1}(-1)$. The function $w(z)$ is entire and satisfies $w(-1)=-1$, thus there exists $\delta>0$ small enough such that $w(z)(1-x) \in D_{1}(-1)$ for all $x\in (0,\sigma)$ and 
$z\in D_{\delta}(-1)$. 
Applying Lemma \ref{lemma4} we conclude that the function
$$
z\mapsto \int_0^{\sigma} B(w(z)(1-x)) \phi(x) \d x=\Delta_2(z;\sqrt{\sigma})-\Lambda(w(z);\sigma)
$$
is analytic in $D_{\delta}(-1)$. 
\qed


\section{Proof of Theorem \ref{thm_main}}\label{section_proof_thm1}


We are working under assumption that $f$ is analytic in a neighbourhood of $x_0>0$. We can also assume, without loss of generality, that $f(x_0)=0$, since Gaver-Stehfest approximations are linear in $f$ and they are exact for constant functions.

Our goal is to show that for some $\sigma \in (0,1)$ and $\delta>0$ the function $\Lambda(w(z);\sigma)$ is analytic in $D_{\delta}(-1)$. Once this is established, Theorem \ref{thm3} would imply that the function $\Delta(z)$ is analytic in $D_R(0)$ for some $R>1$ and then Cauchy estimates for derivatives of analytic function would give us the desired result: for every $r\in (1;R)$ we have $|f_n(x_0)|=O(r^{-n})$ as $n\to +\infty$.

We recall that $\phi(x)$ is defined by 
\begin{equation*}
\phi(x)=\frac{1}{\pi \sqrt{1-x}} \int_0^{\frac{\pi}{2}}\tilde f(\sqrt{x} \sin(y)) \d y.
\end{equation*} 
Since $f$ is analytic in the neighbourhood of $x_0$ and satisfies $f(x_0)=0$, the function $\tilde f$ (defined by \eqref{def_tilde_f}) is even and analytic in a neighbourhood of $x=0$ and also satisfies $\tilde f(0)=0$, which implies that the function 
$x\mapsto \tilde f(\sqrt{x} \sin(y))$ is analytic in a neighbourhood of $x=0$. Applying Lemma \ref{lemma4} we conclude that the function $\varphi(x):=\phi(x)/x$ is analytic in a neighbourhood of $x=0$. 

\vspace{0.25cm}
Our problem is now reduced to the following one: {\it  given that $A(u)$ is analytic in $D_{1}(-1)$ and $\varphi(x)$ is analytic in a neighbourhood of $x=0$, prove that there exist $\sigma \in (0,1)$ and $\delta>0$ such that the function
\begin{equation}\label{def_Delta3_2}
\Lambda(w(z))=\int_0^{\sigma} (1+w(z)(1-x))^{-\frac{3}{2}} A(w(z)(1-x)) x \varphi(x) \d x, 
\end{equation}
is analytic in $D_{\delta}(-1)$.  }
\vspace{0.25cm}

Since $\varphi(x)$ is analytic in a neighbourhood of $x=0$, there exists $\epsilon \in (0,1)$ small enough such that the function of two variables
$$
(w,u) \mapsto \varphi((1+w-u^2)/w) 
$$
is analytic in $(w,u) \in D_{\epsilon}(-1)\times D_{\epsilon}(0)$. We set $\sigma=\epsilon^2/4$. Recall
that $w(z)=ze^{z+1}$ and it an entire function that satisfies $w(-1)=-1$. Therefore we can find $\delta \in (0,1)$ small enough such that  
the following two conditions hold
\begin{itemize}
\item[(i)] $w(z) \in D_{\sigma}(-1)$ for $z\in D_{\delta}(-1)$;
\item[(ii)] $1+w(z)=0$ for $z\in D_{\delta}(-1)$ only if $z=-1$. 
\end{itemize}

Note that $w(-1)=-1$, $w'(-1)=0$ and $w''(-1)=1$, thus $1+w(z)=(z+1)^2 \tilde w(z)$ for some function $\tilde w(z)$ with $\tilde w(-1)=1/2$. According to condition (ii) above, $\tilde w(z)\neq 0$ for $z\in D_{\delta}(-1)$. Thus we conclude that the function
$\eta_1(z):=\sqrt{1+w(z)}=(z+1)\sqrt{\tilde w(z)}$ is analytic in $D_{\delta}(-1)$.  It is also  clear that 
$\eta_1(z)\in D_{\epsilon}(0)$ for $z\in D_{\delta}(-1)$. 

From condition (i) above we find that 
\begin{equation}\label{eta_2_bound1}
1+w(z)(1-\sigma) \in D_{(1-\sigma) \sigma}(\sigma) \;\;\; 
{\textrm{ for }} \; z\in D_{\delta}(-1),
\end{equation}
The fact that $0 \notin D_{(1-\sigma) \sigma}(\sigma)$ implies $1+w(z)(1-\sigma)\neq 0$ for $z\in D_{\delta}(-1)$, so that the function
$\eta_2(z):=\sqrt{1+w(z)(1-\sigma)}$ is analytic and nonzero in $D_{\delta}(-1)$. From \eqref{eta_2_bound1} we also conclude that 
\begin{equation}\label{eta_2_bound2}
|\eta_2(z)|\le \sqrt{\sigma+(1-\sigma) \sigma}< \sqrt{2\sigma}=\frac{\epsilon}{\sqrt{2}}<\epsilon
\;\;\; 
{\textrm{ for }} \; z\in D_{\delta}(-1),
\end{equation}
thus $\eta_2(z) \in D_{\epsilon}(0)$ for $z\in D_{\delta}(-1)$.

Assume now that $z\in (-1,-1+\delta)$, so that $1+w(z)\in (0,\sigma)$ and 
$1+w(z)(1-\sigma) \in (\sigma,\sigma+(1-\sigma)\sigma)$. We change the variable of integration $x\mapsto u=\sqrt{1+w(1-x)}$ 
in \eqref{def_Delta3_2}, so that $x=(1+w-u^2)/w$ and obtain
\begin{align}\label{eqn_Delta3_a}
\nonumber
\Lambda(w(z);\sigma)&=\int_0^{\sigma} (1+w(z)(1-x))^{-\frac{3}{2}} A(w(z)(1-x)) x \varphi(x) \d x
\\  &=
\frac{2}{w(z)^2} \int_{\eta_2(z)}^{\eta_1(z)}\Big(\frac{1+w(z)}{u^2} -1\Big) K(w(z),u) \d u
\end{align}
where we defined
$$
K(w,u):=A(u^2-1) \varphi((1+w-u^2)/w).
$$ 
Since $A$ is analytic in $D_{1}(-1)$ and $\varphi((1+w-u^2)/w)$ is analytic in $(w,u) \in D_{\epsilon}(-1)\times D_{\epsilon}(0)$, we conclude that 
the function $K(w,u)$ is analytic in
$(w,u) \in D_{\epsilon}(-1)\times D_{\epsilon}(0)$. 

Next, we define 
$$
L(w,u):=\frac{1+w}{u^2} \big[ K(w,u)-K(w,0)\big]-K(w,u). 
$$
Since the function $u\mapsto K(w,u)$ is even and analytic in
$(w,u) \in D_{\epsilon}(-1)\times D_{\epsilon}(0)$ we conclude that the function $u\mapsto (K(w,u)-K(w,0))/u^2$ is analytic in $u \in D_{\epsilon}(-1) $ for each $w \in D_{\epsilon}(0)$, thus the function $L(w,u)$ is analytic in $(w,u) \in D_{\epsilon}(-1)\times D_{\epsilon}(0)$. 
Therefore, there exists a function $M(w,u)$ that is analytic in $(w,u) \in D_{\epsilon}(-1)\times D_{\epsilon}(0)$ and satisfies 
$$
\frac{\d }{\d u} M(w,u)=L(w,u), \;\;\; {\textnormal{for}} \;\; (w,u) \in D_{\epsilon}(-1)\times D_{\epsilon}(0).
$$
With these definitions of $L$ and $M$ we can rewrite the integrand in the right-hand side of \eqref{eqn_Delta3_a} as follows
$$
\Big(\frac{1+w}{u^2} -1\Big) K(w,u)=\frac{1+w}{u^2}K(w,0)+L(w,u)=\frac{1+w}{u^2}K(w,0)+\frac{\d }{\d u} M(w,u).
$$
Now we can evaluate the integral in \eqref{eqn_Delta3_a}: 
\begin{align*}
& \int_{\eta_2(z)}^{\eta_1(z)} \Big[\frac{1+w}{u^2} K(w,0)+\frac{\d }{\d u} M(w,u) \Big] \d u \\
&=  \Big[ -\frac{1+w}{u} K(w,0)+M(w,u)\Big] \Big|^{u=\eta_1(z)}_{u=\eta_2(z)} \\
&=  \Big(-\frac{1+w}{\eta_1(z)}+\frac{1+w}{\eta_2(z)}\Big)  K(w,0) + M(w,\eta_1(z))-M(w,\eta_2(z)),
\end{align*}
so that we finally obtain (using the fact that $1+w(z)=\eta_1^2(z)$)
\begin{equation}\label{eqn_Delta3_final}
\Lambda(w(z);\sigma)=\frac{2}{w(z)^2} \times \Big[ \Big(-\eta_1(z)+\frac{1+w(z)}{\eta_2(z)}\Big) K(w(z),0) + M(w(z),\eta_1(z))-M(w(z),\eta_2(z)) \Big].
\end{equation}
So far we have established \eqref{eqn_Delta3_final} for $z\in (-1,-1+\delta)$. However, due to our choice of 
$\sigma$ and $\delta$, the right-hand side in \eqref{eqn_Delta3_final} is an analytic function 
of $z\in D_{\delta}(-1)$, which proves that the function $\Lambda(w(z);\sigma)$ can be extended to an analytic function in  $z\in D_{\delta}(-1)$. 
This ends the proof of Theorem \ref{thm_main}. 
\qed


\section{Proof of Theorem \ref{thm2}}\label{section_proof_thm2}


We are working under assumption that $f$ is $m$ times differentiable at $x_0>0$ and $f(x_0)=0$. We can also assume, without loss of generality, that 
$f^{(j)}(x_0)=0$ for $j=1,\dots,m$. Indeed, the Taylor expansion of $f$ at $x_0$ gives us
$$
f(x)=\sum_{k=1}^m \frac{f^{(k)}(x_0)}{k!} (x-x_0)^k + h_m(x)(x-x_0)^m= P(x) +R(x),
$$
where $h_m(x)\to 0$ as $x\to x_0$. Since Gaver-Stehfest approximations are linear, we have
$$
f_n(x)=P_{n}(x)+R_{n}(x),
$$
where  $P_n(x)$, and $R_n(x)$ are the $n$-th Gaver-Stehfest approximations of $P(x)$ and $R(x)$, respectively. The function $P$ is a polynomial, in particular it is analytic and thus Theorem 
\ref{thm_main} implies that $P_n(x_0)$ converge to $0=P(x_0)$ exponentially fast as $n\to +\infty$. Therefore, $f_n(x_0)=o(n^{-k})$ as $n\to +\infty$ if and only if 
$R_n(x_0)=o(n^{-k})$. 

Next, we argue that Theorem \ref{thm2} will be established if we can show that for some $\sigma>0$ and $\delta>0$ the function $\frac{\d^k}{\d z^k} \Lambda(w(z);\sigma)$ is bounded in $D_{\delta}(-1) \cap D_1(0)$. Assuming this result, Theorem \ref{thm3} implies that the function
$\Delta^{(k)}(z)$ is continuous on ${\overline {D_1(0)}}\setminus \{-1\}$ and bounded in ${\overline {D_1(0)}}$. From \eqref{def_Delta} we find 
$$ 
\Delta^{(k)}(z)= \sum_{n\geq k} n(n-1)\cdots(n-k+1)f_n(x_0)(-1)^n z^{n-k}, \;\;\; |z|<1.
$$
Thus, for any $0<r<1$ and $n\geq k$, we have
\begin{align*}
n(n-1)\cdots(n-k+1)f_n(x_0)(-1)^{n}&=\frac{1}{2\pi i} \int_{|z|=r} z^{-(n-k)-1}\Delta^{(k)}(z) \d z\\
&= r^{-(n-k)}\int_0^1 e^{-2\pi i(n-k)t}\Delta^{(k)}(re^{2\pi it}) \d t.
\end{align*}
Taking the limit as $r \uparrow 1$ (and using the Dominated Convergence Theorem) we conclude that 
\begin{align*}
n(n-1)\cdots(n-k+1)f_n(x_0)(-1)^{n-k}=\int_0^1 e^{-2\pi i(n-k)t}\Delta^{(k)}(e^{2\pi it}) \d t.
\end{align*}
Since $\Delta^{(k)}(e^{2\pi it})$ is continuous and bounded on $(0,1/2)\cup (1/2,1)$, it follows from the Riemann-Lebesgue lemma that 
$$
n(n-1)\cdots(n-k+1)f_n(x_0)\to 0, \;\;\; n\to +\infty, 
$$
which is equivalent to $f_n(x_0)=o(n^{-k})$. 

Next, we recall that $\phi(x)$ is defined via \eqref{def_phix}. Since 
$f^{(j)}(x_0)=0$ for $j=0,1,\dots,m$, we also have $\tilde f^{(j)}(x_0)=0$ for $j=0,1,\dots,m$ (see \eqref{def_tilde_f}), thus $\tilde f(x)=O(x^{m})$ as $x\to 0$ and therefore  
$\phi(x)=O(x^{m/2})$ as $x\downarrow 0$.

\vspace{0.25cm}
Our problem is now reduced to the following one: {\it given that $m=2k+3$, $A(u)$ is analytic in $D_{1}(-1)$ and
$\phi(x)$ is an integrable function on $(0,1-\epsilon)$ (for any $\epsilon>0$) 
that satisfies $\phi(x)=O(x^{m/2})$ as $x\downarrow 0$, prove that there exist $\sigma \in (0,1)$ and $\delta>0$ such that the function
\begin{equation}\label{def_Delta3_3}
\frac{\d^k}{\d z^k}\Lambda(w(z);\sigma)=\frac{\d^k}{\d z^k}\int_0^{\sigma} (1+w(z)(1-x))^{-\frac{3}{2}} A(w(z)(1-x)) \phi(x) \d x, 
\end{equation}
is bounded in $D_{\delta}(-1) \cap D_1(0)$. }
\vspace{0.25cm}

For  $\delta \in (0,1)$ we define 
$$
\Omega_{\delta}:=\{ w\in {\mathbb C} \; : \; w=ze^{z+1}, \; z\in D_{\delta}(-1) \cap D_1(0)\}.
$$
On Figure \ref{fig2} we plot the domains $D_{\delta}(-1) \cap D_1(0)$ and $\Omega_{\delta}$ for $\delta=1/2$.

\begin{figure}
\centering
\subfloat[]{\label{fig2_p1}\includegraphics[height =6cm]{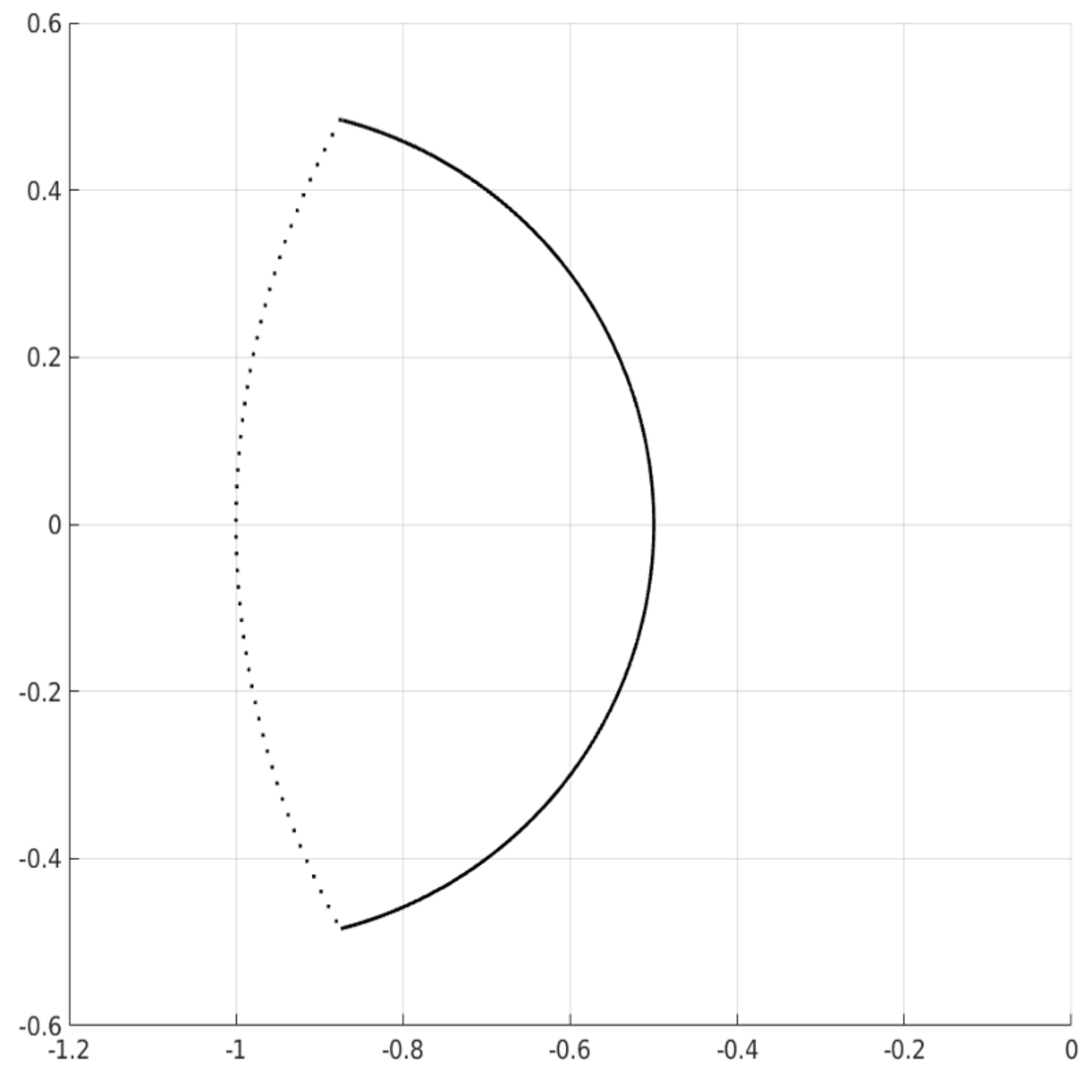}} 
\qquad
\subfloat[]{\label{fig2_p2}\includegraphics[height =6cm]{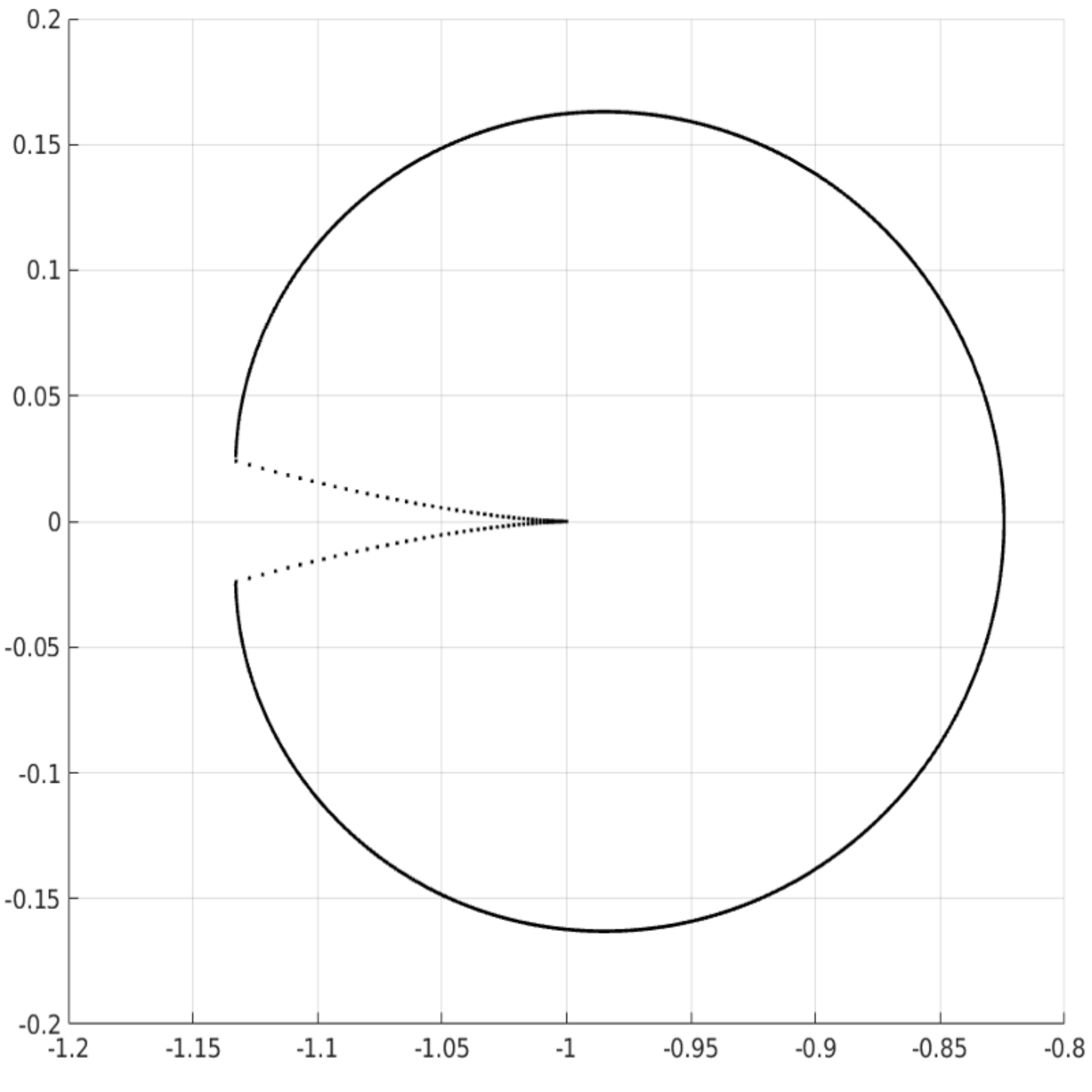}}
\caption{(a) Domain $D_{\frac{1}{2}}(-1) \cap D_1(0)$ and (b) its image $\Omega_{\frac{1}{2}}$ under the map $z\mapsto w=z e^{z+1}$} 
\label{fig2}
\end{figure}

Before we can proceed with the proof of Theorem \ref{thm2}, we need to establish three auxiliary results.  

\begin{lemma}\label{lemma11}
${}$
\begin{itemize}
\item[(i)] As $z\to -1$ we have $w(z)=-1+\frac{1}{2}(z+1)^2+\frac{1}{3}(z+1)^3+O((z+1)^4)$.
\item[(ii)] Let $w\in \Omega_{\delta}$ and $1+w=a+\i b$ for real $a$ and $b$. If $a<0$ then $b^2>C|a|^{3}$ for some positive constant $C=C(\delta)$.
\item[(iii)] Let $w\in \Omega_{\delta}$ and $(1+w)/(-w)=a+\i b$ for real $a$ and $b$. If $a<0$ then $b^2>C|a|^{3}$ for some positive constant $C=C(\delta)$.
\end{itemize}
\end{lemma}
\begin{proof}
Part (i) follows by Taylor expansion of $w(z)=ze^{z+1}$. 
To prove part (ii), we parametrize the circle  $|z|=1$ as $z(u)=-\cos(u)-\i \sin(u)$, so that $u=0$ corresponds to $z=-1$. 
Writing Taylor series near $u=0$ we see that 
$$
z(u)=-1+\frac{u^2}{2}-\i u + O(u^3), 
$$ 
and using the result in item (i) we compute
$$
1+w(z(u))=\frac{1}{2}\Big( \frac{u^2}{2}-\i u \Big)^2+\frac{1}{3}\Big( \frac{u^2}{2}-\i u \Big)^3+O(u^4)=-\frac{1}{2} u^2 -\i \frac{1}{6} u^3+O(u^4).
$$
Thus we see that the boundary of the domain ${\Omega}_{\delta}$ near $w=-1$ (that is represented by the dotted line on figure \ref{fig2_p2}) is paramaterized by the curve 
$\gamma(u)=-1-\frac{1}{2} u^2 -\i \frac{1}{6} u^3+O(u^4)$ near $u=0$. Equivalently, if $\gamma(u)=-1+a+\i b$, then we have parametrization $b^2=\frac{2}{9}|a|^3+o(|a|^3)$ near the point $w=-1$. Thus if we take $C>0$ small enough, then the entire curve 
$\{ z=-1+x+\i y \; : \; x<0, \; y \in {\r}, \; y^2=C |x|^3\}$ will lie outside of the domain ${\Omega_{\delta}}$. This ends the proof of item (ii).

Item (iii) follows from (i) and (ii). 
\end{proof}

Next, we define
$$
Q_{\alpha,\beta}(w):=\int_0^{\sigma} |1+w(1-x)|^{-\alpha} x^{\beta} \d x,
$$
where $\alpha>0$, $\beta>0$, $\sigma \in (0,1)$ and $w\in {\mathbb C}\setminus (-\infty,-1]$.

\begin{lemma}\label{lemma12}
Assume that $\beta>0$ and $0\le \gamma < \alpha$. The function $w\mapsto |1+w|^{\gamma}Q_{\alpha,\beta}(w)$ is bounded in
${\Omega}_{\delta}$ if $\beta+1\ge \max(\alpha, 3\alpha/2-\gamma)$.
\end{lemma}
\begin{proof}
First we need to bound from below the value of $|1+w(1-x)|$, for $x\in (0,\sigma)$ and $w\in {\Omega_{\delta}}$. For $s$ and $t$ ranging over some subsets of $(0,\infty)$ we will write 
$s \approx t$ if for some positive constants $C_1$ and $C_2$ we have $C_1 t < s < C_2 t$ for all $s$ and $t$. Thus, for  $w\in \Omega_{\delta}$ we have $|w| \approx 1$ and
$$
|1+w(1-x)|=|w| \times |(1+w)/(-w)+x| \approx |(1+w)/(-w)+x|. 
$$ 
Let $(1+w)/(-w)=a+\i b$ for real $a$ and $b$. It is clear that $a=O(1)$ and $b=O(1)$ when $w\in {\Omega_{\delta}}$. If $a>-x/2$ then $x+a>x/2$ and we have an inequality
$$
|(1+w)/(-w)+x|^2=|(x+a)+\i b|^2=(x+a)^2+b^2>x^2/4+b^2. 
$$
If $a\le -x/2$ (so that $a<0$ and $x\le 2|a|$) we have 
$$
|(1+w)/(-w)+x|^2=|(x+a)+\i b|^2=(x+a)^2+b^2\ge b^2. 
$$
Thus, there exists a constant $C>0$ such that   
\begin{align*}
Q_{\alpha,\beta}(w)&=\int_0^{\sigma} |1+w(1-x)|^{-\alpha} x^{\beta} \d x <C \int_0^{\sigma} |(1+w)/(-w)+x|^{-\alpha} x^{\beta} \d x\\
&<C \times \bigg [ {\mathbf{1}}_{\{a<0\}}\int_0^{2|a|} |b|^{-\alpha} x^{\beta} \d x+ \int_0^{\sigma} (x^2/4+b^2)^{-\alpha/2} x^{\beta} \d x \bigg].
\end{align*}
We have 
$$
\int_0^{2|a|} |b|^{-\alpha} x^{\beta} \d x=O(|a|^{\beta+1} |b|^{-\alpha}).
$$
Performing change of variables $x=2|b|y$ we compute
\begin{align*}
I:=\int_0^{\sigma} (x^2/4+b^2)^{-\alpha/2} x^{\beta} \d x=2^{\beta+1} |b|^{\beta+1-\alpha} 
\int_0^{\sigma/|b|} (1+y^2)^{-\alpha/2} y^{\beta} \d y.
\end{align*}
If $\sigma/|b|\le 1$ the integral in the right-hand side of the above equation is $O(1)$, and since $|b|=O(1)$ and $\beta+1\ge \alpha$ we conclude that in this case $I=O(1)$. If $\sigma/|b|>1$, we write  
$$
\int_0^{\sigma/|b|} (1+y^2)^{-\alpha/2} y^{\beta} \d y=\int_0^{1} (1+y^2)^{-\alpha/2} y^{\beta} \d y+\int_1^{\sigma/|b|} (1+y^2)^{-\alpha/2} y^{\beta} \d y.
$$
The first integral is a constant (depending only on $\alpha$ and $\beta$). In the second integral, the integrand can be bounded from above and below by a constant multiple of $y^{\beta-\alpha}$. Thus, when $\sigma/|b|>1$, the second integral can be estimated as 
$$
\int_1^{\sigma/|b|} (1+y^2)^{-\alpha/2} y^{\beta} \d y \approx \int_1^{\sigma/|b|} y^{\beta-\alpha} \d y= O(1)+ O(|b|^{\alpha-\beta-1}).
$$
Combining these results we obtain an estimate (in the case $\sigma/|b|>1$)
$$
I=\int_0^{\sigma} (x^2/4+b^2)^{-\alpha/2} x^{\beta} \d x= |b|^{\beta+1-\alpha}  \times ( O(1)+ O(|b|^{\alpha-\beta-1}))=O(|b|^{\beta+1-\alpha})+O(1)=O(1), 
$$
where in the last step we again used the fact that $|b|=O(1)$ and $\beta+1\ge \alpha$.

It is clear that $|1+w|^{\gamma}=O(1)$ in ${\Omega_{\delta}}$. Thus, combining the above estimates, we conclude
\begin{equation}\label{eqn31}
|1+w|^{\gamma} Q_{\alpha,\beta}(w)=O(1)+{\mathbf{1}}_{\{a<0\}}O(|1+w|^{\gamma} |a|^{\beta+1} |b|^{-\alpha}).
\end{equation}
For $w \in \Omega_{\delta}$ we have
$$
|1+w|=O(|(1+w)/(-w)|)=O((a^2+b^2)^{1/2})=O(|b|(1+(a/b)^2)^{1/2}).
$$
When $a<0$ we have $|b|^{-1}=O(|a|^{-3/2})$ (see Lemma \ref{lemma11}(iii)), thus we obtain 
\begin{equation}\label{eqn32}
|1+w|=O(|b| (1+|a|^{-1})^{1/2})=O(|b| \times |a|^{-1/2}). 
\end{equation}
From \eqref{eqn31} and \eqref{eqn32} (and using $|b|^{-1}=O(|a|^{-3/2})$) we estimate for $a<0$
\begin{equation*}
|1+w|^{\gamma} |a|^{\beta+1} |b|^{-\alpha}=|a|^{\beta+1-\gamma/2} |b|^{-\alpha+\gamma}=|a|^{\beta+1-\gamma/2} 
|a|^{-3/2(\alpha-\gamma)}=|a|^{\beta+\gamma+1-3\alpha/2}
\end{equation*}
and this latter quantity is bounded since $|a|=O(1)$ and $\beta+1\ge 3\alpha/2-\gamma$. 
\end{proof}

We leave to the reader the proof of the next result: 
it can be done by induction or using Faa di Bruno's formula. 
\begin{lemma}\label{lemma_Bruno}
For every $k\in\mathbb{N}$ there exist polynomials 
 $\{P_{k,j}(x_1,\ldots,x_k)\}_{1\le j \le k}$ such that for any smooth functions $g$ and $h$
\begin{equation}\label{psi_expansion}
\frac{\d^k}{\d z^k} g(h(z)) = \sum_{j=1}^{k} (h'(z))^{\max(2j-k,0)} \times g^{(j)}(h(z))  \times  P_{k,j}(h'(z),\ldots,h^{(k)}(z)).
\end{equation}
\end{lemma}

Now we are ready to complete the proof of Theorem \ref{thm2}. We recall that all that is left to do is to establish the fact stated (in italic font) on page
\pageref{def_Delta3_3}. To simplify notation, we define $\psi(w):= (1+w)^{-3/2}A(w).$
With this notation we have
$$
\frac{\d^k } {\d z^k}\Lambda(w(z);\sigma)=\int_0^{\sigma} \frac{\d^k } {\d z^k}\psi(w(z)(1-x)) \phi(x) \d x.
$$
Invoking Lemma \ref{lemma_Bruno}, we have
\begin{align}\label{eqn33}
&\frac{\d^k } {\d z^k}\Lambda(w(z);\sigma)
= \sum_{j=1}^{k} (w'(z))^{\max(2j-k,0)} \\
\nonumber
& \qquad \times \int_0^{\sigma}\psi^{(j)}(w(z)(1-x)) (1-x)^{\max(2j-k,0)}
 P_{k,j}(w'(z)(1-x),\dots,w^{(k)}(z)(1-x)) \phi(x) \d x.
\end{align}
The function $A(u)$ is analytic in $D_{1}(-1)$. We choose $\sigma>0$ and $\delta>0$ small enough so that 
$w(1-x) \in D_{\frac{1}{2}}(-1)$ for $w\in {\Omega_{\delta}}$ and $x\in (0,\sigma)$ and $|\phi(x)|<C_1 x^{m/2}$ for some $C_1>0$ and all $x\in (0,\sigma)$. 
We compute 
$$
\psi^{(j)}(w(1-x))=\sum\limits_{l=0}^j \binom{j}{l} \times \bigg[ \prod\limits_{i=0}^{l-1} (-3/2-i) \bigg] (1+w(1-x))^{-3/2-l} A^{(j-l)}(w(1-x)).
$$
The terms $A^{(j-l)}(w(1-x))$ are bounded for $w\in {\Omega_{\delta}}$ and $x\in (0,\sigma)$. 
Thus 
$$
|\psi^{(j)}(w(1-x))|=O(|1+w(1-x)|^{-3/2-j}), \;\;\; w\in {\Omega_{\delta}}, \;\;\; x\in (0,\sigma). 
$$
The functions $P_{k,j}(w'(z)(1-x),\dots,w^{(k)}(z)(1-x))$ are bounded for $z\in D_{\delta}(-1) \cap D_1(0)$ and $x\in (0,\sigma)$, since $P_{k,j}$ is a polynomial and $w$ an entire function. We observe that 
$w'(z)=(z+1)e^{z+1}=w(1+z)/z$. This fact coupled with  the result 
$1+w(z)=\frac{1}{2}(z+1)^2+O((z+1)^3)$ (that was proved earlier in Lemma 
\ref{lemma11}) implies that $|w'(z)|=O(|1+w(z)|^{1/2})$ in $D_{\delta}(-1) \cap D_1(0)$. 
Combining all these observations and using \eqref{eqn33} we conclude that there exists $C_2>0$ such that for all $z\in D_{\delta}(-1) \cap D_1(0)$ 
\begin{align}\label{final_eqn}
\Big|\frac{\d^k } {\d z^k}\Lambda(w(z);\sigma)\Big|&<
C_2 \sum_{j=1}^{k} |1+w(z)|^{\max(j-k/2,0)} \int_0^{\sigma}|1+w(z)(1-x)|^{-3/2-j} \phi(x) \d x
\\
\nonumber &=C_1 \times C_2 \sum_{j=1}^{k} |1+w(z)|^{\max(j-k/2,0)} Q_{3/2+j, m/2}(w(z)).
\end{align}
We leave it to the reader to check that if $m=2k+3$ then for all $j=1,2,\dots,k$
$$
m/2+1\ge 3/2+j \qquad {\textnormal{ and }} \qquad m/2+1+\max(j-k/2,0) \ge (3/2) \times (3/2+j).
$$
According to Lemma \ref{lemma12}, each term $|1+w|^{\max(j-k/2,0)} Q_{3/2+j, m/2}(w)$ in \eqref{final_eqn} is bounded when $w\in {\Omega_{\delta}}$, thus $\frac{\d^k } {\d z^k}\Lambda(w(z);\sigma)$ is bounded 
in $D_{\delta}(-1) \cap D_1(0)$.

\qed

\section*{Acknowledgements}
Research was supported by the Natural Sciences and Engineering Research Council of Canada.


\end{document}